\documentclass [a4paper, 12pt] {article}
\usepackage{amsfonts}
\usepackage{bm}
\usepackage{amsmath}
\usepackage{enumerate}
\usepackage{enumitem}
\usepackage{graphicx,color}
\usepackage{epstopdf}
\usepackage{float}
\usepackage{amsopn}

\newtheorem{theorem}{Theorem}[section]

\newtheorem{remark}[theorem]{Remark}
\newtheorem{lemma}[theorem]{Lemma}

\newtheorem{example}[theorem]{Example}
\newenvironment{proof}[1][Proof]{\noindent\textbf{#1.} }{\ \rule{0.5em}{0.5em}}

\begin{document}
\begin{center}
    {\bf \Large A semi-discrete numerical method for convolution-type unidirectional wave equations}
    \\ ~ \\ \vspace*{20pt}
    H. A. Erbay\textsuperscript{1}, S. Erbay\textsuperscript{1},  A. Erkip\textsuperscript{2}
    \vspace*{20pt}

    {$^{1}$Department of Natural and Mathematical Sciences, Faculty of Engineering, Ozyegin University,  Cekmekoy 34794, Istanbul, Turkey}

    \vspace*{20pt}

    {$^{2}$Faculty of Engineering and Natural Sciences, Sabanci University, Tuzla 34956, Istanbul,  Turkey}
\end{center}

{\let\thefootnote\relax\footnote{{E-mail:   husnuata.erbay@ozyegin.edu.tr, saadet.erbay@ozyegin.edu.tr,
                    albert@sabanciuniv.edu}}}

\begin{abstract}
Numerical approximation of a general class of nonlinear unidirectional wave equations with a convolution-type nonlocality in space is considered. A  semi-discrete numerical method  based on both  a uniform space discretization and the discrete convolution operator is introduced to solve the Cauchy problem. The method is proved to be uniformly convergent as the mesh size goes to zero. The order of convergence for the discretization error  is linear or quadratic depending on the smoothness of the convolution kernel. The discrete problem defined on the whole spatial domain is then truncated  to a finite  domain. Restricting the problem to a finite domain introduces a localization error and  it is proved that this localization error  stays below a given threshold if the finite domain is large enough. For two particular kernel functions, the numerical examples concerning solitary wave solutions illustrate the expected accuracy of the method. Our class of nonlocal  wave equations includes the Benjamin-Bona-Mahony equation as a special case and the present work is inspired by the previous work of Bona, Pritchard and Scott on numerical solution of the Benjamin-Bona-Mahony equation.
\end{abstract}

\vspace*{20pt}


\setcounter{equation}{0}
\section{Introduction}\label{sec:sec1}

In this paper, we propose a semi-discrete numerical approach based on a uniform spatial discretization and truncated discrete convolution sums for the computation of solutions to  the Cauchy problem associated to the one-dimensional nonlocal nonlinear wave equation, which is a regularized conservation law,
\begin{equation}
     u_{t} +\left(\beta \ast  f(u)\right)_{x}=0,  \label{eq:cont}
\end{equation}%
with a general kernel function $\beta(x)$ and  the convolution integral
 \begin{displaymath}
    (\beta \ast v)(x)= \int_{\mathbb{R}} \beta(x-y)v(y)\mbox{d}y.
\end{displaymath}
We prove  error estimates showing  the first-order or second-order convergence, in terms of the mesh size,  depending on the smoothness of the kernel function. Also, our numerical experiments  confirm the theoretical results.

Members of the class (\ref{eq:cont}) arise  as  model equations in  different contexts of physics, from shallow water waves to elastic deformation waves in  dense lattices.   For instance, in the case of  the exponential kernel  $\beta(x)={1\over 2}e^{-|x|}$ (which is the Green's function for the differential operator $1-D^{2}_{x}$ where $D_{x}$ represents the partial derivative with respect to $x$) and $f(u)=u+u^{p+1}$ with  $p\geq 1$,  (\ref{eq:cont}) reduces to  a generalized form of the Benjamin-Bona-Mahony (BBM) equation \cite{Benjamin1972},
\begin{equation}
    u_{t}+u_{x}-u_{xtt}+(u^{p+1})_{x}=0,  \label{eq:bbm}
\end{equation}
that has been widely used  to model unidirectional surface water waves with small amplitudes and long wavelength.  On the other hand, if the kernel function $\beta$ is chosen as the  Green's function of the differential operator $1+D^{4}_{x}$,
\begin{equation}
    \beta(x)= {1\over {2\sqrt 2}}e^{-{\vert x\vert\over \sqrt 2}}\Big( \cos\big({{\vert x\vert}\over {\sqrt 2}}\big) + \sin \big({{\vert x\vert}\over {\sqrt 2}}\big) \Big ), \label{eq:ros-ker}
\end{equation}
and if $f(u)=u+g(u)$,  (\ref{eq:cont}) reduces to a generalized form of the  Rosenau equation \cite{Rosenau1988}
\begin{equation}
    u_{t} +u_{x}+u_{xxxxt}+(g(u))_{x}=0 \label{eq:rosenau}
\end{equation}
that has received much attention as a propagation model for  weakly nonlinear long waves on  one-dimensional dense crystal lattices within a quasi-continuum framework. It is worth to mention here that, in   general,  $\beta$ does not have to be  the Green's function of a differential operator. In other words, (\ref{eq:cont}) cannot always be transformed into a partial differential equation and those members might be called "genuinely nonlocal". Naturally, in such a case,  standard finite-difference schemes  will not be applicable to those equations. In this work we consider a numerical scheme based on truncated discrete convolution sums, which also solves  such genuinely nonlocal equations.

We  remark that the Fourier transform  of the kernel function gives the exact dispersion relation between the phase velocity and wavenumber of infinitesimal waves in linearized theory. In other words, general dispersive properties of waves are represented by the kernel function. Obviously the waves are nondispersive when the kernel function is the Dirac delta function. That is, (\ref{eq:cont}) can be viewed as a regularization of  the hyperbolic conservation law $u_{t}+(f(u))_{x}=0$, where the convolution integral plugged into the conservation law is the only source of dispersive effects. Naturally, our motivation in developing the present numerical scheme also stems from the need to understand the interaction between nonlinearity and nonlocal dispersion.

Recently,  two different approaches have been proposed in \cite{Borluk2017,Erbay2018} to solve numerically
the nonlocal nonlinear bidirectional wave equation  $u_{tt}=\left(\beta \ast  f(u)\right)_{xx}$. In \cite{Borluk2017} the authors have developed a semi-discrete pseudospectral Fourier method and they have proved the convergence of the method for a general  kernel function. By pointing out that, in most cases,  the kernel function is given  in physical space rather than Fourier space, in \cite{Erbay2018}  the present authors have developed a semi-discrete scheme based on spatial discretization, that can be directly applied to the bidirectional wave equation with  a general arbitrary kernel function.   They have proved  a semi-discrete error estimate for the scheme and have investigated, through numerical experiments,   the relationship between the blow-up time and the kernel function for the solutions blowing-up in finite time. These motivate us to apply a similar approach to the initial-value problem of (\ref{eq:cont}) and  to develop a convergent semi-discrete scheme that can be directly applied to  (\ref{eq:cont}). To the best of our knowledge, no efforts have been made yet to solve numerically (\ref{eq:cont}) with a general  kernel function.

As in  \cite{Erbay2018} our strategy  in obtaining  the  discrete problem is to transfer the spatial derivative in (\ref{eq:cont}) to the kernel function and to discretize the convolution integral on a uniform grid. Thus, the spatial discrete derivatives of $u$ do not appear in the resulting discrete problem. Depending on the smoothness of the kernel function, the two error estimates, corresponding to the first- and second-order accuracy in terms of the mesh size, are established for the spatially discretized solution. If  the exact solution  decreases fast enough in space,  a truncated discrete model with a finite number of degrees of freedom (that is, a finite number of grid points) can be used and, in such a case, the number of modes depends on the accuracy desired. Of course, this is another source of error in the numerical simulations and it depends on the decay behavior of the exact solution. Following the idea in \cite{Bona1981} we are able to prove a decay estimate for the exact solution under  certain conditions on the kernel function.  We also address the above issues in two model problems; propagation of  solitary waves for both the BBM equation and  the Rosenau equation.

A numerical scheme based on the discretization of an integral representation of the solution was used in \cite{Bona1981} to solve the BBM equation which is a member of the class (\ref{eq:cont}). The starting point of our numerical method is similar to that in  \cite{Bona1981}. As it was already observed in \cite{Bona1981}, an advantage of this direct approach is that  a further time discretization will not involve any stability issues regarding spatial mesh size. This is due to the fact that  our approach does not involve any spatial derivatives of the unknown function.

The paper is structured as follows. In Section \ref{sec:sec2} we focus on  the continuous Cauchy problem for (\ref{eq:cont}).  In Section \ref{sec:sec3},  the semi-discrete problem obtained by discretizing in space is presented and a short proof of the local well-posedness theorem is given. In Section \ref{sec:sec4} we investigate the convergence of the discretization error with respect to mesh size;   we prove that the convergence rate is linear or quadratic depending on the smoothness of  $\beta$. Section \ref{sec:sec5} is devoted to  analyzing the key properties of the truncation error  arising when we consider only a finite number of grid points. In Section \ref{sec:sec6} we carry out a set of numerical experiments for two specific kernels to illustrate  the theoretical results.

The notation used in the present paper is as follows.  $\Vert u\Vert_{L^p}$ is the $L^p$ ($1\leq p \leq \infty$) norm of $u$ on $\mathbb{R}$, $W^{k,p}(\mathbb{R})=\{ u\in L^p(\mathbb{R}): D^ju \in L^p(\mathbb{R}),~~ j \leq k \} $ is the $L^{p}$-based Sobolev space with the norm $\Vert u\Vert_{W^{k,p}}=\sum_{j\leq k} \|D^j u\|_{L^p},~1\leq p \leq \infty$ and  $H^{s}$ is the usual $L^{2}$-based  Sobolev space of index $s$ on $\mathbb{R}$.  $C$ denotes a generic positive constant. For a real number $s$, the symbol $[s]$ denotes the largest integer less than or equal to $s$.

\setcounter{equation}{0}
\section{The Continuous  Cauchy Problem}\label{sec:sec2}

We  consider the Cauchy problem
\begin{align}
  &  u_{t} +(\beta \ast f(u))_{x}=0, \text{ \ \ \ }  x\in \mathbb{R}\text{, \ \ }t>0, \label{eq:cont1}  \\
  &  u(x,0)=\varphi(x), \text{ \ \ \ }~~~~~~~~  x\in \mathbb{R}.  \label{eq:initial}
\end{align}%
We assume that $f$ is sufficiently smooth with $f(0)=0$ and that the kernel $\beta $ satisfies:
\begin{enumerate}[label={\bf Assumption \arabic*.}, align=left]
 {\it    \item $\beta \in L^{1}(\mathbb{R})$,

    \item $\beta ^{\prime }=\mu $ is a finite Borel measure on $\mathbb{R}$. }
\end{enumerate}

We note that Condition 2 above also includes the more regular case $\beta \in W^{1,1}(\mathbb{R})$ in which case
$d\mu=\beta^{\prime }dx$. The following theorem deals with the local well-posedness of (\ref{eq:cont1})-(\ref{eq:initial}).
\begin{theorem}\label{theo:theo2.1}
    Suppose that $\beta$ satisfies Assumptions 1 and 2. Let $s>\frac{1}{2}$, $f\in C^{[s] +1}(\mathbb{R})$ with $f(0)=0$. For a given $\varphi \in H^{s}(\mathbb{R})$, there is some $T>0$ so that the initial-value problem (\ref{eq:cont1})-(\ref{eq:initial}) is locally well-posed with solution $u\in C^{1}\left([0,T],H^{s}(\mathbb{R})\right)$.
\end{theorem}
The proof of Theorem \ref{theo:theo2.1} follows from Picard's theorem  for Banach space-valued ODEs. The nonlinear term $f(u)$ is locally Lipschitz on $H^{s}(\mathbb{R})\cap L^{\infty }(\mathbb{R})$  \cite{Duruk2010}.  Moreover, the conditions on $\beta$ imply that the term $\left(\beta \ast  f(u)\right)_{x}=\beta^{\prime} \ast  f(u)$ maps $H^s(\mathbb{R})$ onto itself. Hence, (\ref{eq:cont1}) is an $H^{s}$-valued ODE.

We will later use the following estimate on the nonlinear term \cite{Constantin2002}.
\begin{lemma}\label{lem:lem2.2}
    Let $s\geq 0,$ $f\in C^{[s]+1}(\mathbb{R})$ with $f(0)=0$. Then for any $u\in H^{s}(\mathbb{R})\cap L^{\infty }(\mathbb{R})$, we have $f(u)\in H^{s}(\mathbb{R})\cap L^{\infty }(\mathbb{R})$. Moreover there is some constant $C(M)$ depending on $M$ such that for all $u\in H^{s}(\mathbb{R})\cap L^{\infty }(\mathbb{R})$ with $\left\Vert u\right\Vert _{L^{\infty} }\leq M$
    \begin{displaymath}
        {\left\Vert f(u)\right\Vert }_{H^{s}}\leq C(M){\left\Vert u\right\Vert }_{H^{s}}~.
    \end{displaymath}
\end{lemma}
We emphasize that the bound in the lemma depends only on the $L^{\infty}$-norm of $u$. This in turn allows us to control the $H^{s}$-norm of the solution by its $L^{\infty}$-norm. In particular, finite time blow-up of solutions is independent of regularity and is controlled only by the $L^{\infty}$-norm of $u$.
\begin{lemma}\label{lem:lem2.3}
    Suppose the conditions of Theorem \ref{theo:theo2.1} are satisfied and $u\in C^{1}\left([0,T],H^{s}(\mathbb{R})\right)$ is the solution of  (\ref{eq:cont1})-(\ref{eq:initial}). Then
    \begin{equation}
        {\left\Vert u(t)\right\Vert }_{H^{s}}\leq \left\Vert \varphi \right\Vert_{H^{s}}e^{Ct},~~~~0\leq t\leq T,
        \label{growth23}
    \end{equation}
    where $C$ depends on  $M=\sup_{0\leq t\leq T}\left\Vert u(t)\right\Vert_{L^{\infty}}$ and $\beta$.
\end{lemma}
\begin{proof}
    Integrating (\ref{eq:cont1}) with respect to time, we get
    \begin{equation}
        u(x,t)=\varphi(x)-\int_{0}^{t}\big(\beta^{\prime}\ast f(u)\big)(x,\tau)d\tau.
    \end{equation}
    Using Young's inequality  and Lemma \ref{lem:lem2.2},  we obtain
    \begin{equation}
        \Vert u(t)\Vert_{H^s}\leq \Vert \varphi \Vert_{H^{s}}+C(M, \beta)\int_{0}^{t}\Vert u(\tau)\Vert_{H^{s}} d\tau
    \end{equation}
    for all $t\in [0, T]$.  By  Gronwall's lemma this gives (\ref{growth23}).
\end{proof}
\begin{remark}
It is well-known that, under suitable convexity assumptions, the hyperbolic conservation law $u_{t}+(f(u))_{x}=0$ leads to shock formation in a finite time even for smooth initial conditions.  On the other hand, according to Lemma \ref{lem:lem2.3}, for any solution $u$ of  (\ref{eq:cont}) the derivatives will stay bounded as long as $u(t)$ stays bounded in the $L^{\infty}$-norm. In other words, the regularization prevents a shock formation.
\end{remark}

\setcounter{equation}{0}
\section{The  Discrete  Problem}\label{sec:sec3}

In this section we  first give two lemmas for error estimates of discretizations of integrals and derivatives on an infinite uniform grid, respectively, and then introduce the discrete problem associated to  (\ref{eq:cont1})-(\ref{eq:initial}).

\subsection{Discretization and Preliminary Lemmas}

Consider doubly infinite sequences $\mathbf{w}=(w_{i})_{i=-\infty}^{i=\infty}=(w_{i})$ of real numbers  $w_{i}$ with $i\in\mathbb{Z}$  (where $\mathbb{Z}$ denotes the set of integers). For a fixed $h>0$ and $1\leq p<\infty $, the  $l_{h}^{p}(\mathbb{Z})$ space is defined as
\begin{displaymath}
   l_{h}^{p}\left(\mathbb{Z}\right)=\left\{ (w_{i}): w_{i}\in \mathbb{R}, ~~
        \Vert \mathbf{w}\Vert_{l_{h}^{p}}^{p}=\sum_{i=-\infty }^{\infty}h|w_{i}|^{p}\right\}.
\end{displaymath}
The $l^{\infty}(\mathbb{Z})$ space with the sup-norm $\displaystyle \Vert \mathbf{w}\Vert_{l^{\infty}}=\sup_{i \in\mathbb{Z}} \left \vert w_{i} \right \vert$ is a Banach space.  The discrete convolution operation denoted by the symbol $*$ transforms two sequences  $\mathbf{w}$ and $\mathbf{v}$ into a new sequence:
\begin{equation}
    (\mathbf{w}\ast \mathbf{v})_{i}=\sum_{j}hw_{i-j}v_{j} \label{eq:disc-con}
\end{equation}%
(henceforth, we use $\displaystyle \sum_{j}$ to denote summation over all $j\in\mathbb{Z}$). Young's inequality for convolution integrals state that $\Vert \mathbf{w}\ast \mathbf{v}\Vert_{l_{h}^{p}}\leq \Vert \mathbf{w}\Vert _{l_{h}^{1}}\Vert \mathbf{v}\Vert _{l_{h}^{p}}$ for
$\mathbf{w}\in l_{h}^{1}, \mathbf{v}\in l_{h}^{p}$, $1\leq p <\infty $ and $\Vert \mathbf{w}\ast \mathbf{v}\Vert _{l^{\infty} }\leq \Vert \mathbf{w}\Vert_{l_{h}^{1}}\Vert \mathbf{v}\Vert_{l^{\infty}}$ for $\mathbf{w}\in l_{h}^{1},\mathbf{v}\in l^{\infty }$.

Consider a function $w$ of one variable $x$  defined on $\mathbb{R}$. We then introduce a uniform partition of the real line with the mesh size $h$ and with the grid points $x_{i}=ih$, $i\in \mathbb{Z}$. Let the restriction operator $\mathbf{R}$ be   $\mathbf{R}w=(w(x_{i}))$.  We will henceforth use the abbreviations $\mathbf{w}$ and $\mathbf{w}^{\prime }$ for $\mathbf{R}w$ and  $\mathbf{R}w^{\prime }$,  respectively.

The following lemma gives the error bounds for the discrete approximations of the integral over $\mathbb{R}$ depending on the smoothness of the integrand. We note that the two cases  correspond to the rectangular and trapezoidal approximations, respectively.
\begin{lemma}\label{lem:lem3.1} ~
    \begin{enumerate}[label={\bf \upshape(\alph*)}]
    \item
        Let $w\in L^{1}(\mathbb{R})$ and  $w^{\prime}=\mu$ be a finite measure  on $\mathbb{R}$. Then
        \begin{equation}
            \left\vert \int_\mathbb{R} w(x)dx-\sum_{i}h w(x_{i})\right \vert \leq h\vert \mu\vert (\mathbb{R}). \label{eq:intA}
        \end{equation}%
    \item
        Let $w\in W^{1,1}(\mathbb{R})$ and  $w^{\prime\prime}=\nu$ be a finite measure  on $\mathbb{R}$. Then
        \begin{equation}
            \left\vert \int_\mathbb{R} w(x)dx-\sum_{i}h w(x_{i})\right \vert \leq h^{2}\vert \nu\vert (\mathbb{R}). \label{eq:intB}
        \end{equation}
    \end{enumerate}
\end{lemma}
The following lemma handles the $l^{\infty}$ estimates for the discrete approximation of the first derivative  $w^{\prime}$, whose proof follows more or less  standard lines.
\begin{lemma}\label{lem:lem3.2} ~
    Let $\mathbf{w}=\mathbf{R}w$ and $\mathbf{w}^{\prime }=\mathbf{R}w^{\prime }$. Let $D$ be the discrete  derivative operator defined by the central differences
    \begin{equation}
    (D\mathbf{w})_{i}={1\over 2h}(w_{i+1}-w_{i-1}), ~~~~i\in\mathbb{Z}.
    \end{equation}
    \begin{enumerate}[label={\bf \upshape(\alph*)}]
    \item
        If $w\in W^{2,\infty}(\mathbb{R})$, then
        \begin{equation}
            \Vert D\mathbf{w}-\mathbf{w}^{\prime }\Vert_{l^{\infty}}\leq \frac{h}{2}
                \Vert w^{\prime\prime}\Vert_{L^{\infty}} .  \label{eq:lemma32a}
        \end{equation}%
    \item
        If $w\in W^{3,\infty}(\mathbb{R})$, then
        \begin{equation}
            \Vert D\mathbf{w}-\mathbf{w}^{\prime }\Vert_{l^{\infty}}\leq \frac{h^{2}}{6}
                \Vert w^{\prime\prime\prime}\Vert_{L^{\infty}}. \label{eq:lemma32b}
        \end{equation}%
    \end{enumerate}
\end{lemma}
We refer the reader to \cite{Erbay2018} for the proofs of the two lemmas above.

\subsection{The Semi-Discrete Problem}

In order to get the semi-discrete problem associated with (\ref{eq:cont1})-(\ref{eq:initial}) we discretize them in space with a fixed mesh size $h>0$.  Thus, the discretized form of the nonlocal wave equation (\ref{eq:cont1}) becomes
\begin{equation}
    \frac{d \mathbf{v}}{dt} =-D\big(\bm{\beta}_{h}\ast f(\mathbf{v})\big)  \label{eq:disc}
\end{equation}%
with the notation  $f(\mathbf{v})=(f(v_{i}))$ and $\bm{\beta}_{h}=\mathbf{R}\beta$. The identity $D(\bm{\beta}_{h}\ast f(\mathbf{v}))=(D \bm{\beta}_{h})\ast f(\mathbf{v})$ allows us to transfer the discrete derivative on the kernel. So in order to prove the local well-posedness theorem for the semi-discrete problem, we need merely prove the following lemma which estimates
the discrete derivative $D\bm{\beta}_{h}$ of the restriction $\bm{\beta}_{h}$ of the kernel.
\begin{remark}\label{rem:rem3.2} ~
    We note that (\ref{eq:disc}) involves point values of $\beta$. When $\beta$ satisfies Assumptions 1 and 2 one should pay attention to how the point values in  (\ref{eq:disc}) are defined. To clarify this issue we will assume throughout that
    \begin{equation}
    \beta(x)=\int_{(-\infty, x]} d\mu.   \label{eq:bbb}
    \end{equation}
\end{remark}
\begin{lemma}\label{lem:lem3.3}
    Let $\beta^{\prime}=\mu$ be a finite measure on $\mathbb{R}$. Then $D\bm{\beta}_{h}\in l_{h}^{1}$ and $\Vert D\bm{\beta}_{h}\Vert _{l_{h}^{1}}\leq |\mu| (\mathbb{R}).$
\end{lemma}
\begin{proof}
    The assumption (\ref{eq:bbb}) allows one to write
    \begin{align}
        h(D\bm{\beta}_{h})_{i}
        &={1\over 2}\big(\beta (x_{i+1})-\beta (x_{i-1})\big )\\
        & ={1\over 2}\int_{(x_{i-1}, x_{i+1}]} d\mu
        \leq{1\over 2}\vert \mu\vert \big((x_{i-1}, x_{i+1}]\big),
    \end{align}%
    from which we deduce the estimate
    \begin{displaymath}
        \Vert D\bm{\beta}_{h}\Vert_{l_{h}^{1}}
            =\sum_{i}h \left\vert (D\bm{\beta}_{h})_{i}\right\vert
            \leq {1\over 2}\sum_{i} \vert \mu\vert \big((x_{i-1}, x_{i+1}]\big) = \vert \mu\vert \big(\mathbb{R}\big).
    \end{displaymath}%
    \end{proof}

Let $f$  be a locally Lipschitz function with $f(0)=0$. The map $\mathbf{v}\longrightarrow f(\mathbf{v})$ is locally Lipschitz on $l^{\infty}$. Moreover, by Lemma \ref{lem:lem3.3}, the map
\begin{displaymath}
        \mathbf{v}\longrightarrow D\big(\bm{\beta}_{h}\ast f(\mathbf{v})\big)
\end{displaymath}%
is also locally Lipschitz on $l^{\infty}$. By Picard's theorem on Banach spaces, this implies the local well-posedness of the initial-value problem for (\ref{eq:disc}).
\begin{theorem}\label{theo:theo3.4}
    Let $f$  be a locally Lipschitz function with $f(0)=0$. Then the initial-value problem for (\ref{eq:disc}) is locally well-posed for initial data $\mathbf{v}(0)$ in $l^{\infty}$. Moreover there exists some maximal time $T_{h}>0$ so that the problem has unique solution $\mathbf{v}\in C^{1}([0,T_{h}),l^{\infty})$. The maximal time $T_{h}$, if finite, is determined by the blow-up condition
    \begin{equation}
        \limsup_{t\rightarrow T_{h}^{-}}\Vert \mathbf{v}(t)\Vert_{l^{\infty} }=\infty .  \label{eq:blow2a}
    \end{equation}
\end{theorem}

\setcounter{equation}{0}
\section{ Discretization Error}\label{sec:sec4}

Suppose that the function $u\in C^{1}\left([0,T],H^{s}(\mathbb{R})\right)$ with sufficiently large $s$  is the unique solution of the continuous problem  (\ref{eq:cont1})-(\ref{eq:initial}). The discretizations of the initial data $\varphi$ on the uniform infinite grid will be denoted by  $\bm{\varphi}_{h}=\mathbf{R}\varphi$. Let   $\mathbf{u}_{h}\in C^{1}\left([0,T_{h}),l^{\infty}\right)$ be the unique solution of  the discrete problem based on (\ref{eq:disc}) and  the initial data $\bm{\varphi}_{h}$. The aim of this section is to estimate the discretization error defined as $\mathbf{R}u(t)-\mathbf{u}_{h}$. Depending on the conditions imposed on the kernel function $\beta$, we  provide the two different theorems establishing the first- and second-order convergence in $h$. The proofs  follow  similar lines as the corresponding ones in \cite{Erbay2018}.
\begin{theorem}\label{theo:theo4.1}
    Suppose that $\beta$ satisfies Assumptions 1 and 2.     Let $s>\frac{5}{2}$, $ f\in C^{[ s] +1}(\mathbb{R})$ with $f(0)=0$.  Let $u\in C^{1}\left([0,T], H^{s}(\mathbb{R})\right)$ be the solution of the initial-value problem (\ref{eq:cont1})-(\ref{eq:initial}) with $~\varphi \in H^{s}(\mathbb{R})$. Similarly, let  $\mathbf{u}_{h}\in C^{1}\left([0,T_{h}),l^{\infty}\right)$ be the solution of  (\ref{eq:disc}) with initial data $\bm{\varphi}_{h}$.   Let $\mathbf{u}(t)=\mathbf{R}u(t) =(u(x_{i},t))$. Then there is some $h_{0}$ so that for $h\leq h_{0}$, the maximal existence time $T_{h}$ of $\mathbf{u}_{h}$ is at least $T$ and
    \begin{equation}
        \Vert \mathbf{u}(t)-\mathbf{u}_{h}(t)\Vert_{l^{\infty}}=\mathcal{O}(h)    \label{eq:fourone}
    \end{equation}
    for all $t\in \lbrack 0,T \rbrack$.
\end{theorem}
\begin{proof}
    We first let $\displaystyle M=\max_{0\leq t\leq T}\Vert u(t)\Vert _{L^{\infty}}$. Since $\Vert \bm{\varphi }_{h}\Vert _{l^{\infty }}\leq \Vert \varphi \Vert_{L^{\infty }}\leq M$, by continuity there is some maximal time $t_{h}\leq T$ such that $\Vert \mathbf{u}_{h}(t)\Vert _{l^{\infty}}\leq 2M$ for all $t\in \lbrack 0,t_{h}]$. Moreover, by the maximality condition either $t_{h}=T$ or $\Vert \mathbf{u}_{h}(t_{h})\Vert _{l^{\infty}}=2M$. \ At the point $\ x=x_{i}$,  (\ref{eq:cont1}) becomes
    \begin{displaymath}
        u_{t}(x_{i}, t) +\big(\beta \ast f(u)\big)_{x}(x_{i}, t)=0 .\text{ }
    \end{displaymath}%
     Recalling that $\mathbf{u}(t)=\mathbf{R}u(t)$, this becomes $\mathbf{u}^{\prime}(t)=-\mathbf{R}\big(\beta \ast f(u)\big)_{x}(t)$.  A residual term $\mathbf{F}_{h}$ arises from the discretization  of (\ref{eq:cont1}):
     \begin{equation}
       {d \mathbf{u}\over dt}= -D\big(\bm{\beta}_{h}\ast f(\mathbf{u})\big)+\mathbf{F}_{h}, \label{eq:u-F}
    \end{equation}%
     where
     \begin{displaymath}
        \mathbf{F}_{h}=D\big(\bm{\beta}_{h}\ast f(\mathbf{u})\big)-\mathbf{R}\big(\beta \ast f(u)\big)_{x}.
     \end{displaymath}
     The $i$th entry of $\mathbf{F}_{h}$ satisfies
    \begin{align}
        (F_{h})_{i}
        &=\Big(\big( \bm{\beta}_{h}\ast \mathbf{R}f(u)_{x}\big)_{i}-\big(\beta \ast f(u)_{x}\big)(x_{i})\Big)
         +\Big(\big(\bm{\beta}_{h}\ast Df(\mathbf{u})\big)_{i}-\big(\bm{\beta}_{h}\ast \mathbf{R}f(u)_{x}\big)_{i}\Big) \nonumber \\
        &=(F_{h}^{1})_{i}+(F_{h}^{2})_{i}, \label{eq:decomp}
    \end{align}%
    where the variable $t$ is suppressed for brevity.     We start with the  term $(F_{h}^{1})_{i}$. Replacing $f(u)$ by $g$ for convenience, we have
    \begin{equation}
        (F_{h}^{1})_{i} =(\bm{\beta}_{h}\ast \mathbf{g}^{\prime})_{i}-(\beta \ast g^{\prime})(x_{i})
                        =\sum_{j}h\beta (x_{i}-x_{j})g^{\prime}(x_{j})  -\int_{\mathbb{R}} \beta (x_{i}-y)g^{\prime}(y)dy. \label{eq:Fh1}
    \end{equation}%
    Since $\beta \in L^{1}(\mathbb{R})$ and $\beta^{\prime}=\mu$ be a finite measure on $\mathbb{R}$, by (\ref{eq:intA}) of Lemma \ref{lem:lem3.1} we have
    \begin{displaymath}
        \left \vert (F_{h}^{1})_{i}\right \vert \leq h\vert \widetilde{\mu}\vert(\mathbb{R})
    \end{displaymath}%
    where $r^{\prime}=\widetilde{\mu}$ with $r(y)=\beta (x_{i}-y)g^{\prime}(y)$. When $\beta^{\prime}\in L^{1}$,
     we have
     \begin{displaymath}
        \Vert r^{\prime} \Vert_{L^{1}}
            \leq \Vert \beta^{\prime} \Vert_{L^{1}}\Vert g^{\prime}\Vert_{L^{\infty}}
                +\Vert \beta\Vert_{L^{1}}\Vert g^{\prime\prime}\Vert_{L^{\infty}}.
    \end{displaymath}%
    Since $\beta^{\prime}=\mu$ and $g^{\prime}$, $g^{\prime\prime}$ are bounded for $s>5/2$, we have
    \begin{align}
    \vert \widetilde{\mu}\vert(\mathbb{R})
    &\leq \vert \mu \vert(\mathbb{R})\Vert g^{\prime}\Vert_{L^{\infty}}+\Vert \beta\Vert_{L^{1}}\Vert g^{\prime\prime}\Vert_{L^{\infty}}, \\
    & \leq  C \big( \vert \mu \vert(\mathbb{R}) +\Vert \beta\Vert_{L^{1}}\Big)\Vert u\Vert_{H^{s}}
    \end{align}
    where we have used Lemma \ref{lem:lem2.2}.
    Thus
    \begin{displaymath}
        |(F_{h}^{1})_{i}|\leq C h\big(\vert \mu \vert (\mathbb{R})+\Vert \beta\Vert_{L^{1}}\big)\Vert u\Vert_{H^{s}} \leq Ch \Vert u\Vert_{H^{s}}.
    \end{displaymath}%
    For the second term $(F_{h}^{2})_{i}$, again with $g=f(u)$ and $s>5/2$ we have
    \begin{align}
        \left \vert (F_{h}^{2})_{i}\right \vert
        &=\big \vert \left( \bm{\beta}_{h}\ast (D\mathbf{g}-\mathbf{R}g^{\prime})\right)_{i}\big\vert
            \leq C\Vert D\mathbf{g} -\mathbf{g}^{\prime }\Vert_{l^{\infty }}, \nonumber \\
         &  \leq Ch\Vert g^{\prime\prime}\Vert_{L^{\infty}}\leq Ch \Vert g\Vert_{H^{s}}\leq C h\Vert u \Vert_{H^{s}},  \nonumber
    \end{align}%
    where Lemma \ref{lem:lem2.2} and (\ref{eq:lemma32a}) of Lemma \ref{lem:lem3.2} are used.    Combining the estimates for $|(F_{h}^{1})_{i}|$ and $|(F_{h}^{2})_{i}|$, we obtain
    \begin{displaymath}
        \Vert \mathbf{F}_{h}(t)\Vert _{l^{\infty }}\leq Ch \Vert u(t)\Vert_{H^{s}}.
    \end{displaymath}
     We now let $\mathbf{e}(t)=\mathbf{u}(t)-\mathbf{u}_{h}(t)$  be the error term. Then, from (\ref{eq:disc}) and (\ref{eq:u-F}) we have
    \begin{displaymath}
        {d\mathbf{e}(t)\over dt}
        =-D\bm{\beta}_{h}\ast \big(f(\mathbf{u})-f(\mathbf{u}_{h})\big)+\mathbf{F}_{h}, ~~~~   \mathbf{e}(0) =\mathbf{0}.
    \end{displaymath}%
    This implies
    \begin{equation}
    \mathbf{e}(t)= \int_{0}^{t} \Big(-D\bm{\beta}_{h}\ast
            \big(f(\mathbf{u})-f(\mathbf{u}_{h})\big)+\mathbf{F}_{h}\Big)d\tau .    \label{eq:eqe}
    \end{equation}
    By noting that
    \begin{align}
        \Vert -D\bm{\beta}_{h}\ast \left (f(\mathbf{u})-f(\mathbf{u}_{h})\right ) \Vert_{l^{\infty}}
            &\leq  \Vert D\bm{\beta}_{h}\Vert_{l_h^{1}}\Vert f(\mathbf{u})-f(\mathbf{u}_{h})\Vert_{l^{\infty}} \nonumber\\
            & \leq  C \Vert \mathbf{u}-\mathbf{u}_{h}\Vert_{l^{\infty }}  \leq C \Vert \mathbf{e}(t)\Vert_{l^{\infty }},
            \nonumber
    \end{align}
    it follows from (\ref{eq:eqe}) that, for $t\leq t_{h}\leq T$,
    \begin{align}
        \Vert \mathbf{e}(t)\Vert _{l^{\infty }}
           & \leq \sup_{0\leq t\leq T}\Vert \mathbf{F}_{h}(t)\Vert _{l^{\infty }} \int_{0}^{t} d\tau +C\int_{0}^{t}\Vert \mathbf{e}(\tau )\Vert_{l^{\infty }}d\tau, \nonumber  \\
           & \leq C \Big (h t  + \int_{0}^{t}\Vert \mathbf{e}(\tau )\Vert_{l^{\infty }}d\tau \Big).
           \label{eq:errfirst}
    \end{align}%
    Then, by Gronwall's inequality,
    \begin{displaymath}
        \Vert \mathbf{e}(t)\Vert _{l^{\infty }}            \leq C h T             e^{C T}.
    \end{displaymath}
    We observe that the constant $C$ depends  on the bounds   $M=\sup_{0\leq t \leq T}\Vert u(t)\Vert_{L^{\infty}}$, $\Vert \beta \Vert_{L^{1}}$ and $\vert \mu \vert (\mathbb{R})$.  We note that, by Lemma \ref{lem:lem2.3},  $\Vert u(t)\Vert_{H^{s}}$ depends on $M$ and $\beta$. The above inequality, in particular, implies that  $\Vert \mathbf{e}(t_{h})\Vert _{l^{\infty }}<M$ for sufficiently small $h$. Then we have $\Vert \mathbf{u}_{h}(t_{h})\Vert _{l^{\infty }}<2M$ showing that $t_{h}=T_{h}=T$. From the above estimate we get (\ref{eq:fourone}).
\end{proof}
\begin{theorem}\label{theo:theo4.2}
    Let $s>\frac{7}{2}$, $ f\in C^{[ s ] +1}(\mathbb{R})$ with $f(0)=0$. Let $\beta \in W^{1,1}(\mathbb{R})$ and $\beta^{\prime\prime}=\nu$ be a finite measure on $\mathbb{R}$.  Let $u\in C^{1}\left([0,T], H^{s}(\mathbb{R})\right)$ be the solution of the initial-value problem (\ref{eq:cont1})-(\ref{eq:initial}) with $~\varphi \in H^{s}(\mathbb{R})$. Similarly, let  $\mathbf{u}_{h}\in C^{1}\left([0,T_{h}),l^{\infty}\right)$ be the solution of  (\ref{eq:disc}) with initial data $\bm{\varphi}_{h}$.   Let $\mathbf{u}(t)=\mathbf{R}u(t) =(u(x_{i},t))$. Then there is some $h_{0}$ so that for $h\leq h_{0}$, the maximal existence time $T_{h}$ of $\mathbf{u}_{h}$ is at least $T$ and
    \begin{equation}
        \Vert \mathbf{u}(t)-\mathbf{u}_{h}(t)\Vert_{l^{\infty}}=\mathcal{O}(h^{2})    \label{eq:fourtwo}
    \end{equation}
    for all $t\in \lbrack 0,T \rbrack$.
\end{theorem}
\begin{proof}
    The proof follows  that of Theorem \ref{theo:theo4.1} very closely. So here we use the same notation and provide only a brief outline.  The main observation needed is that the only place where the proofs differ is the estimate for $(F_{h}^{1})_{i}$ in (\ref{eq:Fh1}).   Since $\beta \in W^{1,1}(\mathbb{R})$ and $\beta^{\prime\prime}=\nu$ be a finite measure on $\mathbb{R}$, by the second part of Lemma \ref{lem:lem3.1} we have
    \begin{displaymath}
        \left \vert (F_{h}^{1})_{i}\right \vert \leq h^{2}\vert \widetilde{\nu}\vert(\mathbb{R})
    \end{displaymath}%
    where $r^{\prime\prime}=\widetilde{\nu}$ with $r(y)=\beta (x_{i}-y)g^{\prime}(y)$. Formally we can write
    \begin{displaymath}
        r^{\prime\prime }(y)=\beta^{\prime\prime }(x_{i}-y)g^{\prime}(y)-2\beta^{\prime} (x_{i}-y)g^{\prime\prime}(y)
                                +\beta (x_{i}-y)g^{\prime\prime\prime}(y).
    \end{displaymath}%
    Noting that $\beta^{\prime\prime}=\nu$ and $g^{\prime}$, $g^{\prime\prime}$, $g^{\prime\prime\prime}$  are bounded for $s>7/2$ and using Lemma \ref{lem:lem2.2}, we get
    \begin{align}
        \vert \widetilde{\nu}\vert(\mathbb{R})
        &\leq C (\vert \nu \vert(\mathbb{R})\Vert g^{\prime}\Vert_{L^{\infty}}
            +2\Vert \beta^{\prime}\Vert_{L^{1}}\Vert g^{\prime\prime}\Vert_{L^{\infty}}
            +\Vert \beta\Vert_{L^{1}}\Vert g^{\prime\prime\prime}\Vert_{L^{\infty}} \nonumber \\
        & \leq  C\big( \vert \nu\vert (\mathbb{R}) +2\Vert \beta\Vert_{W^{1,1}}\big)\Vert u\Vert_{H^{s}}
    \end{align}
    so that
    \begin{displaymath}
        \vert (F_{h}^{1})_{i}\vert
        \leq C h^{2}\big(\vert \nu \vert (\mathbb{R})+2\Vert \beta\Vert_{W^{1,1}}\big)\Vert u\Vert_{H^{s}}
        \leq Ch^{2}\Vert u\Vert_{H^{s}}.
    \end{displaymath}%
    For the second term $(F_{h}^{2})_{i}$, again with $g=f(u)$ and $s>7/2$ we have
    \begin{align}
        \left \vert (F_{h}^{2})_{i}\right \vert
        &=\big \vert \left( \bm{\beta}_{h}\ast (D\mathbf{g}-\mathbf{R}g^{\prime})\right)_{i}\big\vert
            \leq C\Vert D\mathbf{g} -\mathbf{g}^{\prime }\Vert_{l^{\infty }} \nonumber \\
        &  \leq C h^2\Vert g^{\prime\prime\prime}\Vert_{L^{\infty}}\leq Ch^2 \Vert g\Vert_{H^{s}}\leq C h^2\Vert u \Vert_{H^{s}},  \nonumber
    \end{align}%
    where (\ref{eq:lemma32b}) of Lemma \ref{lem:lem3.2} is used.    Using the estimates for $\vert (F_{h}^{1})_{i}\vert $ and $\vert (F_{h}^{2})_{i}\vert $ in (\ref{eq:decomp}), we obtain
    \begin{displaymath}
        \Vert \mathbf{F}_{h}(t)\Vert_{l^{\infty }}
            \leq Ch^{2} \Vert u\Vert_{H^{s}}.
    \end{displaymath}
     Now (\ref{eq:errfirst}) takes the following form
    \begin{equation}
        \Vert \mathbf{e}(t)\Vert _{l^{\infty }}
            \leq C \Big (h^{2} t  + \int_{0}^{t}\Vert \mathbf{e}(\tau )\Vert_{l^{\infty }}d\tau \Big)
    \end{equation}%
     for $t\leq t_{h}\leq T$.
    Then, by Gronwall's inequality,
    \begin{displaymath}
        \Vert \mathbf{e}(t)\Vert _{l^{\infty }}
            \leq C h^{2} T e^{C T}.
    \end{displaymath}
    Now the constant $C$ depends on  the bounds   $M=\sup_{0\leq t \leq T}\Vert u(t)\Vert_{L^{\infty}}$, $\Vert \beta \Vert_{W^{1,1}}$ and $\vert \nu \vert(\mathbb{R})$.  The rest of the proof follows the same lines as the proof of Theorem \ref{theo:theo4.1}.
\end{proof}

\setcounter{equation}{0}
\section{The Truncated Problem and a Decay Estimate}\label{sec:sec5}

\subsection{The Truncated Problem}
In practical computations, one needs to truncate both the infinite series in (\ref{eq:disc-con}) at a finite $N$  and  the infinite  system of equations in (\ref{eq:disc}) to the system of $2N+1$ equations. After truncating, we obtain from (\ref{eq:disc}) the  finite-dimensional system
\begin{equation}
    \frac{dv_{i}^{N}}{dt}
     =-\sum_{j=-N}^{N}hD\beta (x_{i}-x_{j})f(v_{j}^{N}),\text{\ \ \ \ \ \ }-N\leq i\leq N  \label{eq:trunca}
\end{equation}%
where $v_{i}^{N}$ are the components of a vector valued function $\mathbf{v}^{N}(t)$ with finite dimension $2N+1$. In this section we estimate the truncation error resulting from considering (\ref{eq:trunca}) instead of (\ref{eq:disc}) and give a  decay estimate  for solutions to the initial-value problem (\ref{eq:cont1})-(\ref{eq:initial}) with certain kernel functions.

We start by rewriting (\ref{eq:trunca}) in vector form
\begin{displaymath}
    \frac{d\mathbf{v}^{N}}{dt}=-B^{N}f(\mathbf{v}^{N}),
\end{displaymath}
where $B^{N}$ denotes a matrix with $(2N+1)$ rows and $(2N+1)$ columns, whose typical element is  $b^{N}_{ij}=hD\beta (x_{i}-x_{j})$.  Using  the notation  $\displaystyle \Vert \mathbf{w}\Vert_{l^{\infty}}=\max_{-N \leq i \leq N} \left \vert w_{i} \right \vert$ and assuming that $\beta^{\prime}=\mu$ is a finite measure on $\mathbb{R}$ we get from Lemma \ref{lem:lem3.3}
\begin{displaymath}
    \Vert B^{N}\mathbf{w}\Vert_{l^{\infty}}\leq  |\mu| (\mathbb{R})\Vert \mathbf{w}\Vert_{l^{\infty}}.
\end{displaymath}
As long as $f$ is a bounded and  smooth function, there exists a unique solution of the initial-value problem defined for (\ref{eq:trunca}) over an interval $[0, T^{N})$. Moreover  the blow-up condition
\begin{equation}
        \limsup_{t\rightarrow (T^{N})^{-}} \Vert \mathbf{v}^{N}(t)\Vert_{l^{\infty} }=\infty   \label{eq:blow3}
\end{equation}
is compatible with (\ref{eq:blow2a}) in the discrete problem. Consider the projection  ${\cal T}^{N}\mathbf{v}$ of the solution $\mathbf{v}$ of the semi-discrete initial-value problem associated with (\ref{eq:disc}) onto $\mathbb{R}^{2N+1}$, defined by
\begin{displaymath}
    {\cal T}^{N}\mathbf{v}=(v_{-N},v_{-N+1}, \ldots ,v_{0}, \ldots, v_{N-1},v_{N})
\end{displaymath}
with the truncation operator ${\cal T}^{N}: l^{\infty} \rightarrow \mathbb{R}^{2N+1}$. Our goal is to estimate the truncation error  ${\cal T}^{N}\mathbf{v}-\mathbf{v}^{N}$ and show that, for sufficiently large $N$, $\mathbf{v}^{N}$ approximates the solution $u$ of the continuous problem (\ref{eq:cont1})-(\ref{eq:initial}).
\begin{theorem}\label{theo:theo5.1}
    Let $\mathbf{v}\in C^{1}\left([0,T],l^{\infty}\right)$  be the solution of (\ref{eq:disc}) with initial value $\mathbf{v}(0)$ and let
    \begin{displaymath}
        \delta=\sup \big\{ \left\vert v_{i}(t)\right\vert : t\in \left[ 0,T\right] , \left\vert i\right\vert >N\big\} ~~\mbox{and}~~
         \epsilon (\delta )=\max_{\left\vert z\right\vert \leq \delta } \left\vert f(z)\right\vert .
    \end{displaymath}
    Then for sufficiently small $\epsilon(\delta)$, the solution $\mathbf{v}^{N}$ \ of (\ref{eq:trunca}) with initial  value $\mathbf{v}^{N}(0)={\cal T}^{N}\mathbf{v}(0)$ exists for  times $t\in \lbrack 0,T]$ and
    \begin{displaymath}
        \left\vert v_{i}^{N}(t)-v_{i}(t)\right\vert
            \leq C\epsilon (\delta ),\text{ \ }t\in \lbrack 0,T\rbrack,
    \end{displaymath}%
    for all $-N\leq i \leq N$.
\end{theorem}
\begin{proof}
    We follow the approach in the proof of Theorem \ref{theo:theo4.1}. Taking the components with $-N \leq i \leq N$ of (\ref{eq:disc}) we have
    \begin{align}
    \frac{dv_{i}}{dt}
     &=-\sum_{j=-\infty}^{\infty}hD\beta (x_{i}-x_{j})f(v_{j})  \nonumber \\
     &=-\sum_{j=-N}^{N}hD\beta (x_{i}-x_{j})f(v_{j})+F_{i}^{N} \nonumber
    \end{align}%
    with the residual term
    \begin{displaymath}
    F_{i}^{N}=-\sum_{|j|> N}hD\beta (x_{i}-x_{j})f(v_{j}).
    \end{displaymath}%
    Then  ${\cal T}^{N}\mathbf{v}$ satisfies the system
    \begin{displaymath}
        \frac{d{\cal T}^{N}\mathbf{v}}{dt}=-B^{N}f({\cal T}^{N}\mathbf{v})+\mathbf{F}^{N}
    \end{displaymath}
    with the residual term $\mathbf{F}^{N}=(F_{i}^{N})$. Estimating the residual term we get
     \begin{displaymath}
        \left \vert F_{i}^{N}\right \vert \leq |\mu|(\mathbb{R}) \sup_{|j|> N} \left \vert f(v_{j})\right \vert \leq C \epsilon(\delta).
    \end{displaymath}%
    We set  $\displaystyle M=\max_{0\leq t\leq T}\Vert \mathbf{v}(t)\Vert _{l^{\infty} }$. Since $\Vert\mathbf{v}(0)\Vert _{l^{\infty }}\leq M$, by continuity of the solution $\mathbf{v}^{N}$ of the truncated problem there is some maximal time $t_{N}\leq T$ such that we have   $\Vert \mathbf{v}^{N}(t)\Vert_{l^{\infty}}\leq 2M$ for all $t\in \lbrack 0,t_{N}]$. By the maximality condition either $t_{N}=T$ or $\Vert \mathbf{v}^{N}(t_{N})\Vert_{l^{\infty}}=2M$. We define the error term $\widetilde{\mathbf{e}}={\cal T}^{N}\mathbf{v}-\mathbf{v}^{N}$. Then
    \begin{displaymath}
           \frac{d\widetilde{\mathbf{e}}(t)}{dt}=-B^{N}\big(f({\cal T}^{N}\mathbf{v})-f(\mathbf{v}^{N})\big)+\mathbf{F}^{N}, ~~~~
         \widetilde{\mathbf{e}}(0) =\mathbf{0},
    \end{displaymath}%
    so
    \begin{displaymath}
        \widetilde{\mathbf{e}}(t)=\int_{0}^{t} \Big(-B^{N}\big(f({\cal T}^{N}\mathbf{v})-f(\mathbf{v}^{N})\big)+\mathbf{F}^{N}\Big)d\tau .
    \end{displaymath}%
    Then
    \begin{displaymath}
        \Vert \widetilde{\mathbf{e}}(t)\Vert_{l^{\infty}}\leq |\mu| (\mathbb{R})\int_{0}^{t}\left\Vert \big(f({\cal T}^{N}\mathbf{v})-f(\mathbf{v}^{N})\big)(\tau) \right\Vert_{l^{\infty}}d\tau
        + \int_{0}^{t} \Vert \mathbf{F}^{N}(\tau)\Vert_{l^\infty} d\tau.
    \end{displaymath}%
    But
    \begin{displaymath}
        \left\Vert f({\cal T}^{N}\mathbf{v})-f(\mathbf{v}^{N})\right\Vert_{l^{\infty}}
            \leq C \left\Vert {\cal T}^{N}\mathbf{v}-\mathbf{v}^{N}\right\Vert_{l^{\infty}}.
    \end{displaymath}%
   Putting together, we have
    \begin{displaymath}
        \Vert \widetilde{\mathbf{e}}(t)\Vert _{l^{\infty }}\leq CT\epsilon(\delta) +C \int_{0}^{t}\Vert \widetilde{\mathbf{e}}
        (\tau )\Vert _{l^{\infty }}d\tau,
    \end{displaymath}%
    and by Gronwall's inequality,
    \begin{displaymath}
        \Vert \widetilde{\mathbf{e}}(t)\Vert _{l^{\infty }}\leq C\epsilon(\delta) T e^{CT}.
    \end{displaymath}%
    This, in particular, implies that there is some $\epsilon_{0}$ such that for all $\epsilon(\delta)\leq \epsilon_{0} $ we have $\Vert \widetilde{\mathbf{e}}(t_{N})\Vert _{l^{\infty }}<M$. Then we have     $\Vert \mathbf{v}^{N} (t_{N}) \Vert_{l^{\infty }} < 2M$ showing that $t_{N}=T$  and this completes the proof.
\end{proof}

By combining Theorem \ref{theo:theo5.1} with Theorems \ref{theo:theo4.1} and \ref{theo:theo4.2}, respectively, we now state our main results through the following two theorems.
\begin{theorem}\label{theo:theo5.2}
    Let $s>\frac{5}{2}$, $f\in C^{[ s] +1}(\mathbb{R})$ with $f(0)=0$.   Let $\beta \in L^{1}(\mathbb{R})$ and $\beta^{\prime}=\mu$ be a finite measure on $\mathbb{R}$.  Let $u\in C^{1}\left([0,T],H^{s}(\mathbb{R})\right)$  be the solution  of the initial-value problem (\ref{eq:cont1})-(\ref{eq:initial}) with $\varphi \in H^{s}(\mathbb{R})$.  Then for sufficiently small $h$ and  $\epsilon >0$, there is an $N$ so that  the solution $\mathbf{u}_{h}^{N}$ \ of (\ref{eq:trunca}) with initial  values $\mathbf{u}_{h}^{N}(0)={\cal T}^{N}\bm{\varphi}_{h}$, exists for  times $t\in \lbrack 0,T]$ and
    \begin{equation}
        \Big\vert u(ih,t) -\left(\mathbf{u}_{h}^{N}\right)_{i}(t)\Big\vert
         = {\cal O}\left(h+\epsilon\right),
         \text{ \ }  t\in \lbrack 0,T] \label{eq:fivetwo}
    \end{equation}%
    for all $-N\leq i\leq N$.
\end{theorem}
\begin{theorem}\label{theo:theo5.3}
    Let $s>\frac{7}{2}$, $f\in C^{[ s] +1}(\mathbb{R})$ with $f(0)=0$. Let $\beta \in W^{1,1}(\mathbb{R})$ and $\beta^{\prime\prime}=\nu$ be a finite measure on $\mathbb{R}$.  Let $u\in C^{1}\left([0,T],H^{s}(\mathbb{R})\right)$  be the solution  of the initial-value problem (\ref{eq:cont1})-(\ref{eq:initial}) with $\varphi \in H^{s}(\mathbb{R})$.  Then for sufficiently small $h$ and  $\epsilon >0$, there is an $N$ so that  the solution $\mathbf{u}_{h}^{N}$ \ of (\ref{eq:trunca}) with initial  values $\mathbf{u}_{h}^{N}(0)={\cal T}^{N}\bm{\varphi}_{h}$, exists for  times $t\in \lbrack 0,T]$ and
    \begin{equation}
        \Big\vert u(ih,t) -\left(\mathbf{u}_{h}^{N}\right)_{i}(t)\Big\vert
         = {\cal O}\left(h^{2}+\epsilon\right),
         \text{ \ }  t\in \lbrack 0,T] \label{eq:fivethree}
    \end{equation}%
    for all $-N\leq i\leq N$.
\end{theorem}
\begin{remark}\label{rem:rem5.4}
     Finally we want to consider the stability issue of our numerical scheme. Our approach does not involve discretizations of spatial derivatives of $u$, instead the spatial derivatives act upon the kernel $\beta$. This appears as the coefficient matrix $B^{N}$ in (\ref{eq:trunca}) and Lemma \ref{lem:lem3.3} shows that $B^{N}$ is uniformly bounded with respect to  the mesh size $h$. In other words, the factor $h^{-1}$ that appears in the standard discretization of the first derivatives is not effective in (\ref{eq:trunca}); it is already embedded in the uniform estimate of $B^{N}$. So if we further apply time discretization to (\ref{eq:trunca}) there will be no stability limitation regarding spatial mesh size. Roughly speaking, a fully discrete version of our numerical scheme will be straightforward. This is the main reason why we  prefer to use an ODE solver rather than a fully discrete scheme in the numerical experiments performed in the next section.
\end{remark}

\subsection{A Decay Estimate}
A central question about Theorems \ref{theo:theo5.2} and \ref{theo:theo5.3} is whether we can find a fixed value of the truncation parameter $N$ such that it provides to achieve a desired level of the truncation error $\epsilon$.
For a given level of the truncation error, $\epsilon>0$, the equation
\begin{displaymath}
        \epsilon=\max\big\{ |f(u(x,t))|: |x|\geq Nh, ~~t\in [0, T]\big\}.
    \end{displaymath}
provides  an implicit description of the solution set for $N$. We note that Proposition 5.3 of \cite{Erbay2018} guarantees the existence of such $N$ under the assumption that the solution decays to zero as $\vert x \vert \rightarrow \infty$.  Furthermore, decay estimates for the solution to the initial-value problem may play a central role in obtaining a more explicit relation between $N$ and $\epsilon$. The following lemma  establishes a general decay estimate for the solutions corresponding to certain kernel functions.
\begin{lemma}\label{lem:lem5.5}
    Let $\omega\left(x\right) $ be a positive function such that $\left(\left\vert \beta ^{\prime}\right\vert \ast \omega\right)(x) \leq C\omega(x) $ for all $x\in \mathbb{R}$. Suppose that $\varphi \omega^{-1}\in L^{\infty }\left(\mathbb{R}\right) $. The solution $u\in C^{1}\left([0,T], H^{s}(\mathbb{R})\right)$ of (\ref{eq:cont1})-(\ref{eq:initial}) then satisfies the estimate
    \begin{displaymath}
        \left\vert u(x,t)\right\vert \leq C\omega\left( x\right)
    \end{displaymath}
    for all  $x\in \mathbb{R}$,  $t\in \left[ 0,T\right] $.
\end{lemma}
The proof of this lemma  follows a similar idea used in \cite{Bona1981} and it is almost the same as the proof of Lemma B.1 in Appendix B of \cite{Erbay2018}, so we omit it here for brevity.

Next, as applications of the above lemma, we consider the two example kernel functions  that will be used in the next section.
\begin{example}[The generalized BBM equation]
        \label{ex:ex5.6}
      As it was pointed out in Section \ref{sec:sec1}, the generalized BBM equation (\ref{eq:bbm}) corresponds to the kernel $\beta (x) =\frac{1}{2}e^{-\left\vert x\right\vert }$. Note that $\beta^{\prime}(x)=-\mbox{sign}(x)\beta(x)$ and $\vert \beta^{\prime}(x)\vert=\beta(x)$ for $x\neq 0$. Take $\omega\left( x\right) =e^{-r\left\vert x\right\vert }$ with $0<r<1$. So we have
    \begin{align*}
        \left(\vert \beta^{\prime}\vert \ast \omega\right)(x)
        &=\frac{1}{2}\int_{\mathbb{R}} e^{-\left\vert x-y\right\vert }e^{-r\left\vert y\right\vert }dy
        =\frac{1}{2}e^{-r \vert x \vert}\int_{\mathbb{R}}  e^{-\left\vert x-y\right\vert }e^{r\left( \vert x\vert-\vert y\vert \right) }dy \\
        &\leq \frac{1}{2}e^{-r\left\vert x\right\vert }\int_{\mathbb{R}} e^{-(1-r)\left\vert x-y\right\vert }dy
        =\frac{1}{1-r}e^{-r\left\vert x\right\vert }\leq C\omega(x).
    \end{align*}%
    By Lemma \ref{lem:lem5.5}, for initial data with $\varphi={\cal O}( e^{-r\left\vert x\right\vert })$,  solutions of the BBM equation satisfy the decay estimate
    \begin{displaymath}
        \left\vert u(x,t)\right\vert \leq Ce^{-r\left\vert x\right\vert }, ~~~~~0<r<1
    \end{displaymath}
     for all $t\in \left[ 0,T\right]$.
\end{example}
\begin{example}[The generalized Rosenau equation]
    \label{ex:ex5.7}
     Another member of the class (\ref{eq:cont}) is the generalized Rosenau equation (\ref{eq:rosenau}) corresponding to the kernel (\ref{eq:ros-ker}). Note that, for $x \neq 0$,
     \begin{displaymath}
        \beta^{\prime}(x)=-\frac{1}{2}\mbox{sign} (x)e^{-\frac{\vert x \vert}{\sqrt{2}}}\sin\big(\frac{\vert x \vert}{\sqrt{2}}\big).
     \end{displaymath}
     Again if we take $\omega\left( x\right) =e^{-r\frac{ \vert x\vert}{\sqrt{2}} }$ with any $0<r<1$, we get
    \begin{align*}
        \left(\left\vert \beta^{\prime}\right\vert \ast \omega\right)(x)
        &=\frac{1}{2}\int_{\mathbb{R}} e^{-r\frac{ \vert x-y\vert}{\sqrt{2}} }\vert \sin\big(\frac{ \vert x-y\vert}{\sqrt{2}}\big)\vert
        e^{-r\frac{ \vert y\vert}{\sqrt{2}} } dy \\
        &\leq \frac{1}{2}e^{-r\frac{ \vert x\vert}{\sqrt{2}} }\int_{\mathbb{R}} e^{-(1-r)\frac{ \vert x-y\vert}{\sqrt{2}} }dy         = \frac{\sqrt{2}}{1-r}e^{-r\frac{ \vert x\vert}{\sqrt{2}} } .
    \end{align*}%
     Thus, by   Lemma \ref{lem:lem5.5} we deduce that, for initial data satisfying $\varphi \left( x\right) e^{r\frac{\vert x \vert}{\sqrt{2}} }\in
    L^{\infty }\left( \mathbb{R}\right) $, solutions of the Rosenau equation will satisfy
    \begin{displaymath}
       \left \vert u(x,t)\right \vert \leq Ce^{-r\frac{\left\vert x\right\vert}{\sqrt{2} }}, ~~~~~0<r<1
    \end{displaymath}
     for all $t\in \left[ 0,T\right]$.
\end{example}
\begin{figure}[h!]
    \centering
    \includegraphics[width=0.70\linewidth,scale=1.50,keepaspectratio]{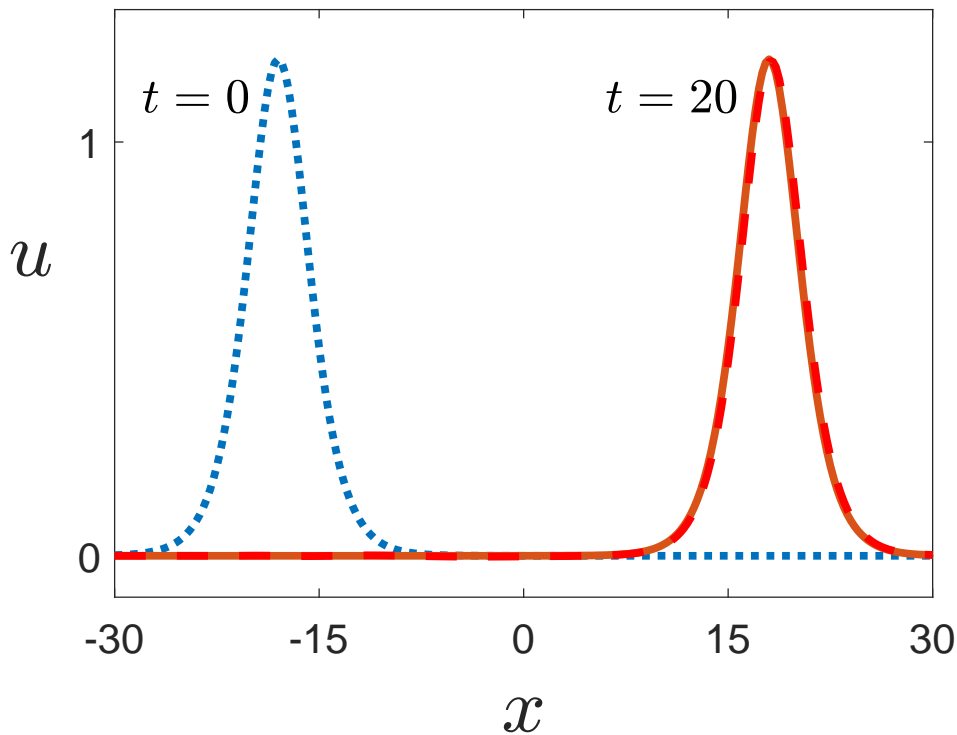}
    \caption{Propagation of a right-moving solitary wave of speed $c=1.8$ for the nonlocal nonlinear wave equation (\ref{eq:cont}) with $\beta(x)={\frac{1}{2}}e^{-|x|}$ and $f(u)=u+u^{2}$. The initial profile, the exact and the numerical solutions at   $t=20$   are shown with the dotted line, the solid line and the dashed line, respectively. The numerical solution is almost indistinguishable from the exact solution. For the numerical simulation the computational domain $[-30, \, 30]$ and the mesh size   $h=0.25$ are used.}
    \label{fig:Fig1}
\end{figure}

\setcounter{equation}{0}
\section{Numerical Experiments}\label{sec:sec6}

In this section we perform  some numerical experiments to compare the numerical solutions of the semi-discrete scheme  developed here with the exact solutions.   We apply the scheme to  two simplest members of the class:  the BBM equation (\ref{eq:bbm}) and   the Rosenau equation (\ref{eq:rosenau}).   For both of the equations, the corresponding kernels satisfy the smoothness condition of  Theorem \ref{theo:theo5.3} and hence we have  the second-order accuracy with respect to the mesh size. Moreover,  Examples  \ref{ex:ex5.6} and   \ref{ex:ex5.7} provide  decay estimates for solutions of these equations which in turn gives us the dependency of $\epsilon$ on $N$ in the estimate of Theorem \ref{theo:theo5.3}.

Except the errors originating from the  temporal  integration of (\ref{eq:trunca}),  there are two sources of  error in the semi-discrete scheme and consequently in our computations:  the discretization of the nonlocal equation in space and the consideration of a finite number of grid points.  This expectation for the semi-discrete scheme is  also confirmed by the numerical experiments. Furthermore,  the numerical experiments in this section  show  that the  spatial discretization error  is of second-order accurate in $h$  and that the  cut-off error resulting from the restriction of the infinite number of the grid points  to a finite number can be made smaller than  the discretization error for all $N$ sufficiently large.
\begin{figure}[h!]
    \centering
    \includegraphics[width=0.70\linewidth,scale=1.50,keepaspectratio]{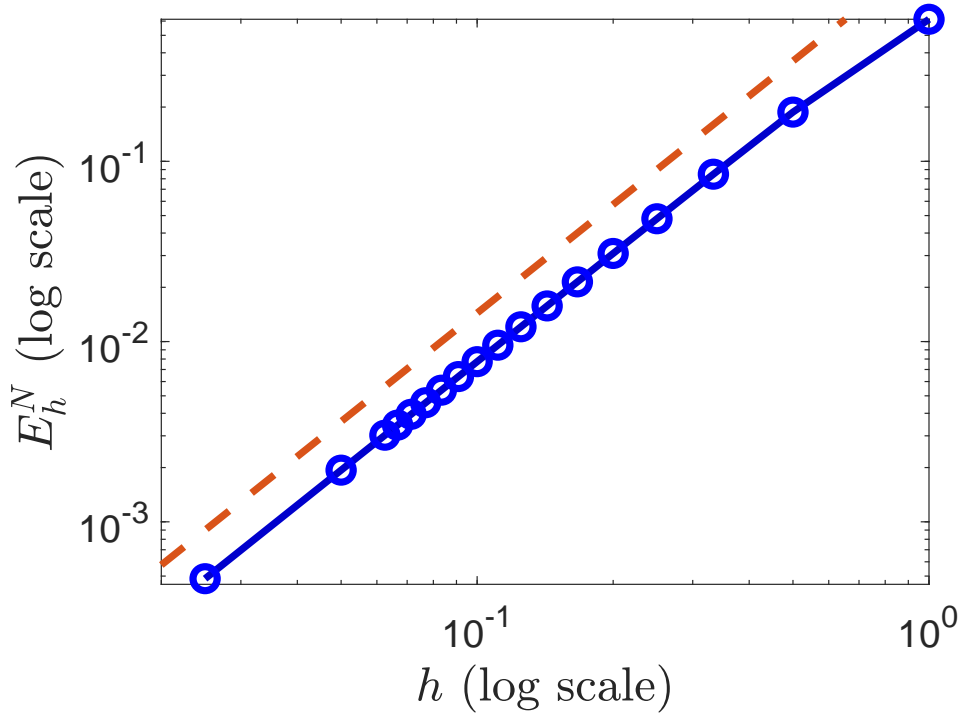}
    \caption{Variation of the error at   $t=20$ with the mesh size $h$. The circle markers indicate the data points of the numerical experiments.  The theoretical quadratic convergence is plotted as a dashed line with slope 2:1 for reference. Propagation of a right-moving solitary wave of speed $c=1.8$ for the nonlocal nonlinear wave equation (\ref{eq:cont}) with $\beta(x)={\frac{1}{2}}e^{-|x|}$ and $f(u)=u+u^{2}$.  The computational domain $[-30, \, 30]$  is used.}
    \label{fig:Fig2}
\end{figure}

In all the experiments, we employ the  Matlab ODE solver \verb"ode45" based on the fourth-order Runge-Kutta method to integrate  (\ref{eq:trunca}) in time.  To keep the temporal errors below the desired tolerance,
we set the relative and absolute tolerances for the solver \verb"ode45" to be $RelTol=10^{-10}$ and $AbsTol=10^{-10}$, respectively. The  $l^{\infty}$-error $E_{h}^{N}$ at time $t$ is calculated as
\begin{equation}
    E_{h}^{N}(t)=\left\Vert \mathbf{u}(t) - \mathbf{u}_{h}^{N}(t) \right\Vert_{l^{\infty}}
                =\max_{-N\leq i \leq N} \left \vert u(x_{i},t)- (\mathbf{u}_{h}^{N})_{i}\right \vert . \label{eq:linferror}
 \end{equation}
By postulating that the relationship between the error and the mesh size is in the form $E_{h}=Ch^{\rho}$, we calculate the  convergence rate $\rho$,
 \begin{equation}
    \rho={{\log\left({{E_{h_{1}}^{N}(t)}/ {E_{h_{2}}^{N}(t)}}\right)}\over {\log\left({h_{1}}/ {h_{2}}\right)}},
 \end{equation}
from the errors at two successive values $h_{1}$ and $h_{2}$ of the mesh size.

\subsection{The  BBM Equation}

We first  illustrate the accuracy of our semi-discrete scheme by considering solitary wave solutions of the generalized BBM equation  which is a member of the class (\ref{eq:cont}) with the exponential kernel.
\begin{figure}[h!]
    \centering
    \includegraphics[width=0.70\linewidth,scale=1.50,keepaspectratio]{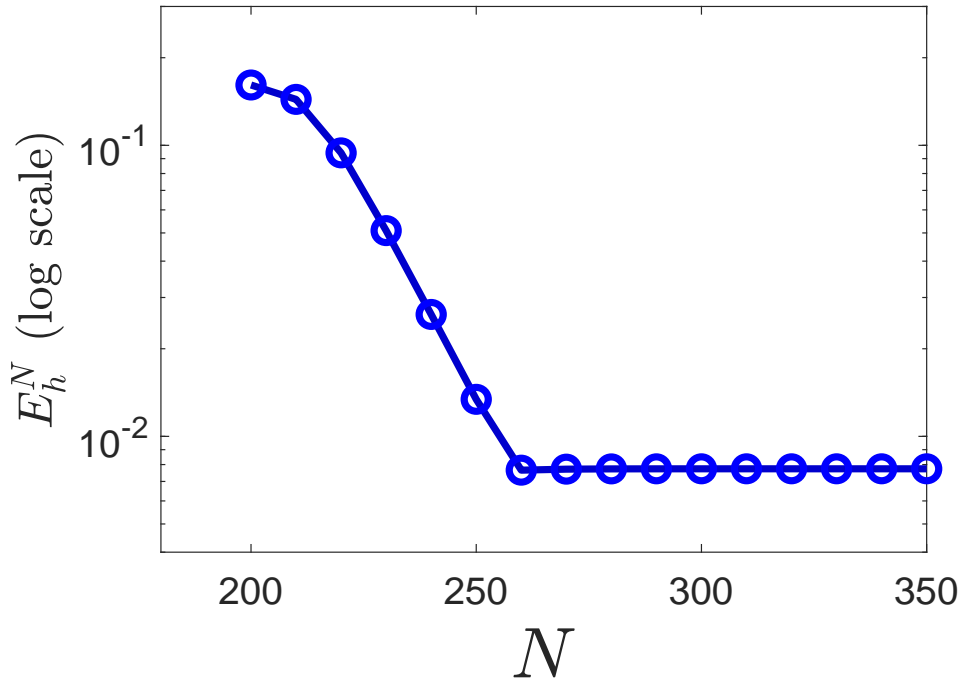}
    \caption{Variation of  the $l^{\infty}$-error ($E_{h}^{N}$) with $N$ for the  solitary wave problem of the nonlocal nonlinear wave equation (\ref{eq:cont}) with the wave speed  $c=1.8$, the kernel $\beta(x)={\frac{1}{2}}e^{-|x|}$  and the quadratic nonlinearity  $f(u)=u+u^{2}$. The  mesh size  is fixed at $h=0.1$. The computational domain is $[-Nh,  Nh]$. The circle markers indicate the data points of the numerical experiments.}
    \label{fig:Fig3}
\end{figure}

The BBM equation (\ref{eq:bbm}) with general power nonlinearity admits  a family of solitary wave solutions given by
\begin{equation}
    u(x,t)=A~\Big(\text{sech}^{2}\big(B(x-ct-x_{0}) \big)\Big)^{1/p}, \label{eq:solitary}
\end{equation}
with $A=\big((p+2)(c-1)/2\big)^{1/p}$, $B=(p/2)(1-1/c)^{1/2}$ and $c >1$ \cite{Bona2000}.  The solitary wave  (\ref{eq:solitary}) is initially located at $x_{0}$ and propagates to the right  with the constant wave speed $c$. To compare the numerical and exact solutions,  we set $p=1$, $c=1.8$ and  solve (\ref{eq:trunca}) with the initial data
\begin{equation}
    u(x,0)=A\, \text{sech}^{2}\big(B (x+18) \big)
\end{equation}
using the Matlab ODE solver \verb"ode45". Here the computational domain is taken as $[-30, 30]$ while  the grid spacing is chosen as $h=0.25$ for which $N=120$. In Figure \ref{fig:Fig1} we plot the exact and numerical solutions at  $ t= 20$  together with the initial state.  The numerical solution has the same amplitude and waveform as that of the exact solution and overall there is a very good agreement between the two curves.  This numerical experiment further verifies that the proposed semi-discrete  scheme  is capable of solving (\ref{eq:cont}) to high accuracy.
\begin{figure}[h!]
    \centering
    \includegraphics[width=0.70\linewidth,scale=1.50,keepaspectratio]{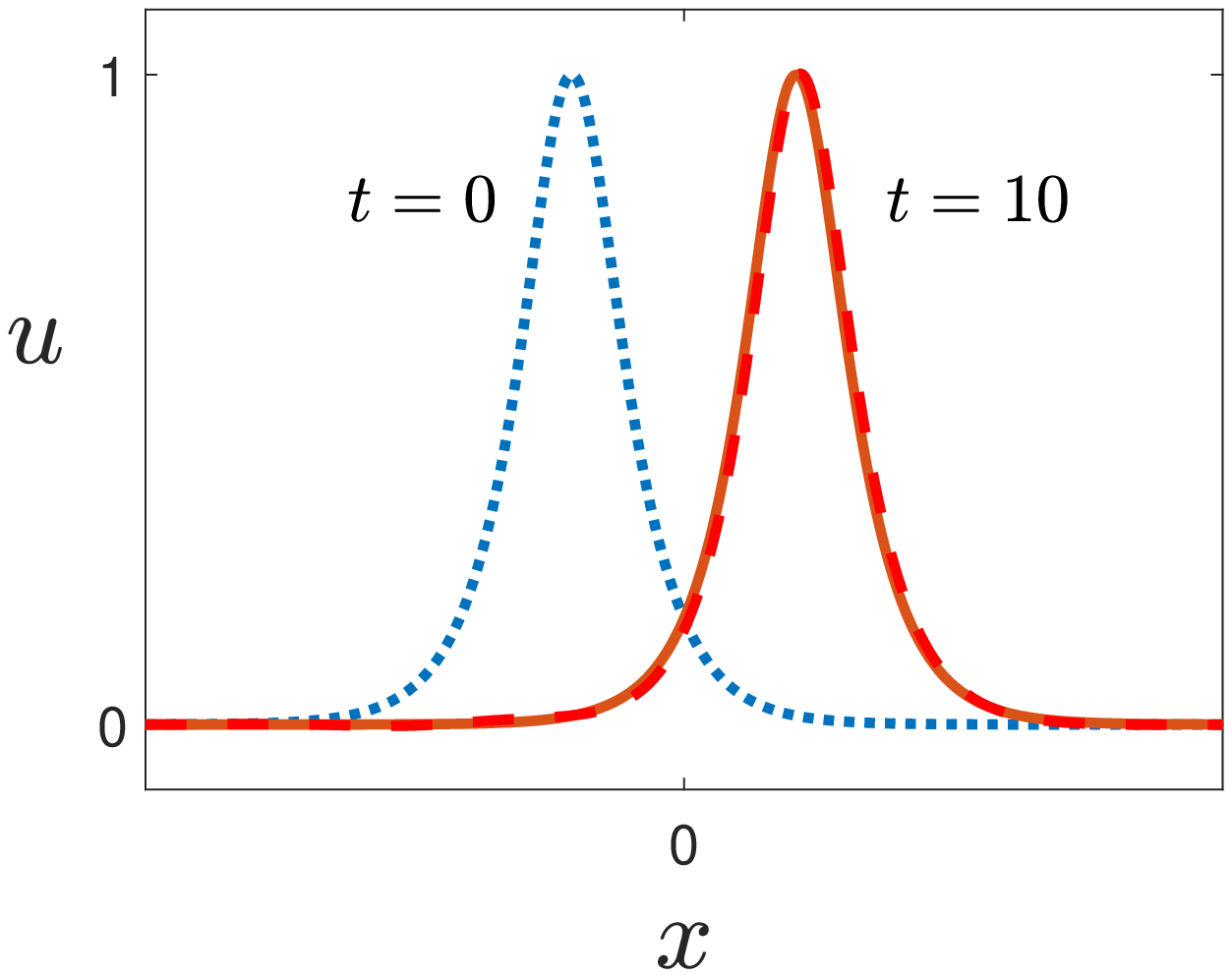}
    \caption{Propagation of a right-moving solitary wave of speed $c=0.5$ for the nonlocal nonlinear wave equation (\ref{eq:cont}) with $\beta(x)= {1\over {2\sqrt 2}}e^{-{\vert x\vert\over \sqrt 2}}\Big( \cos\big({{\vert x\vert}\over {\sqrt 2}}\big) + \sin \big({{\vert x\vert}\over {\sqrt 2}}\big) \Big )$ and $f(u)=u-10u^{3}+12 u^{5}$. The initial profile, the exact and the numerical solutions at   $t=10$   are shown with the dotted line, the solid line and the dashed line, respectively. The numerical solution is almost indistinguishable from the exact solution.
    The computational domain $[-12, \, 12]$ and the mesh size   $h=0.05$ are used.}
    \label{fig:Fig4}
\end{figure}

In order to verify the convergence rate estimate derived in Theorem \ref{theo:theo4.2} for the discretization error we now conduct a sequence of numerical experiments with different mesh sizes.  Figure \ref{fig:Fig2} shows the variation of  the error  measured using (\ref{eq:linferror})  with the mesh size.  The computational domain  is chosen large enough such that the discretization  error (${\cal O}\left(h^{2}\right)$) dominates the truncation error (${\cal O}\left(\epsilon\right)$). The figure has logarithmic scales on both axes and the dashed line corresponding to the theoretical quadratic convergence in space is also displayed  for comparison. From the figure, one can observe that the numerical experiments provide a confirmation of the quadratic convergence rate predicted  by  Theorem \ref{theo:theo4.2}.

In another set of the numerical experiments we aim to show that, for sufficiently large values of $N$, the truncation error  due to the use of a finite number of grid points  has no significant effect on the above numerical results.  First, it is worth emphasizing that the solitary wave  solution in (\ref{eq:solitary}) decays exponentially to zero for $|x| \rightarrow \infty$. Second, Example \ref{ex:ex5.6} shows that $\epsilon={\cal O}\left(e^{-CNh}\right)$ (recall that $E_{h}^{N}={\cal O}\left(h^{2}+\epsilon\right)$). In the present set of the numerical experiments we fix the mesh size and vary the number of grid points, so the size of the computational domain $[-Nh,  Nh]$ is not the same for each experiment. However, by taking  the initial condition and the time interval of the previous experiments and by taking sufficiently large $N$, we guarantee that the wave dynamics occurs over a symmetric interval about the origin (with equidistant from both the left and right endpoints). The variation of the error at time $t=20$ with $N$ is shown in Figure \ref{fig:Fig3} using a semi-logarithmic scale. We see that, up to a certain value of $N$ ($ \approx 260$), the truncation error dominates and it decreases exponentially as $N$ increases as indicated by $\epsilon={\cal O}\left(e^{-CNh}\right)$. However, for larger values of $N$, this phenomenon disappears and the discretization error dominates.

\subsection{The Generalized Rosenau Equation}

We  continue our dicussion  by considering the generalized Rosenau equation (\ref{eq:rosenau}) which is a member of the class (\ref{eq:cont}) with the kernel (\ref{eq:ros-ker}).   In \cite{Park1992} it was pointed out that a particular solitary wave solution
\begin{equation}
    u(x,t)=\text{sech}\big(x-\frac{1}{2}t-x_{0}\big) \label{eq:solitaryR}
\end{equation}
with the constant wave speed $1/2$, the width $1$, the  amplitude $1$ and the initial position $x_{0}$ satisfies the Rosenau equation (\ref{eq:rosenau}) if the nonlinear term is in the form $g(u)=-10u^{3}+12 u^{5}$.  It is worth mentioning that, contrary to the quadratic nonlinearity of the BBM equation, the Rosenau equation considered here has both cubic and quintic nonlinear terms. Consequently, a high resolution enhanced by a smaller mesh size is needed to achieve  the same order of accuracy.

Generally, we follow closely what was done in the previous subsection for the BBM equation. Now, the initial condition is set to $u(x,0)=\text{sech}(x+5/2)$ and  (\ref{eq:trunca}) is solved  up to time $ t= 10$ over the computational domain $[-12, 12]$ corresponding to $h=0.05$ ($N=240$). In Figure \ref{fig:Fig4}   the numerical and exact solutions at $t=10$ are demonstrated and  we  see they agree very well.

To verify the quadratic convergence rate in space we now repeat the same numerical experiment with  different values of $h$ (or equivalently with different values of $N$). On a logarithmic scale, Figure \ref{fig:Fig5} displays the variation of the $l^{\infty}$-error $E_{h}^{N}$ at time $t=10$ with the mesh size $h$  (again together with the dashed line of the theoretical quadratic convergence rate for comparison purposes). The computational results in Figure \ref{fig:Fig5} are very similar to those shown in Figure \ref{fig:Fig2} and they again verify the quadratic  rate of convergence claimed in Theorem \ref{theo:theo4.2} for the discretization error.
\begin{figure}[h!]
    \centering
    \includegraphics[width=0.70\linewidth,scale=1.50,keepaspectratio]{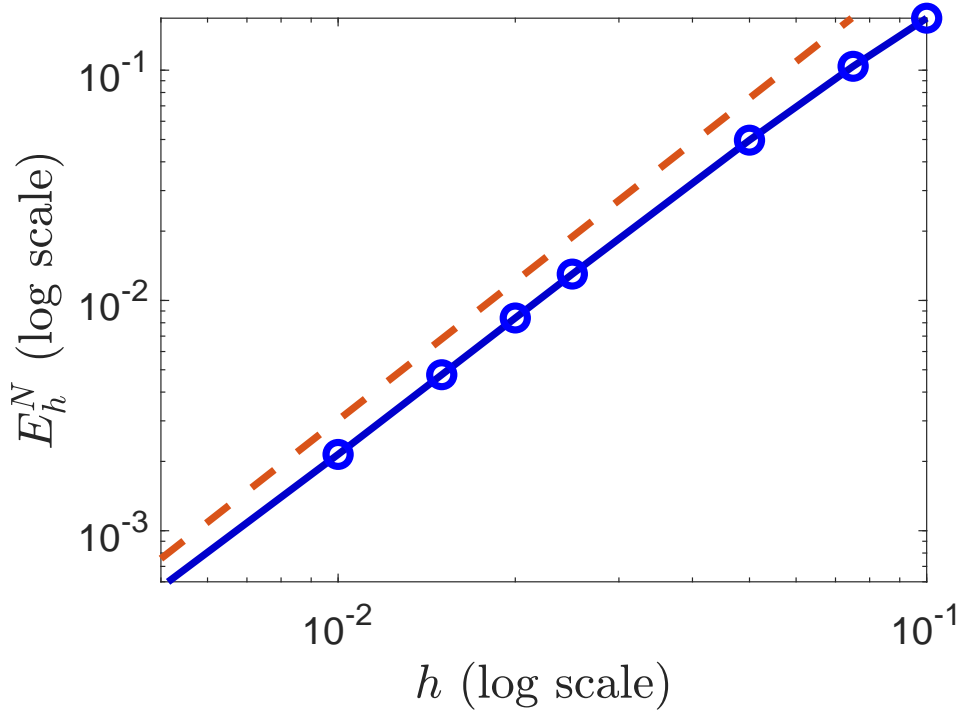}
    \caption{Variation of the error at   $t=10$ with the mesh size $h$. The circle markers indicate the data points of the numerical experiments. The theoretical quadratic convergence is plotted as a dashed line with slope 2:1 for reference. Propagation of a right-moving solitary wave of speed $c=0.5$ for the nonlocal nonlinear wave equation (\ref{eq:cont}) with $\beta(x)= {1\over {2\sqrt 2}}e^{-{\vert x\vert\over \sqrt 2}}\Big( \cos\big({{\vert x\vert}\over {\sqrt 2}}\big) + \sin \big({{\vert x\vert}\over {\sqrt 2}}\big) \Big )$ and $f(u)=u-10u^{3}+12 u^{5}$.  The computational domain $[-12, \, 12]$  is used.}
    \label{fig:Fig5}
\end{figure}

In another set of numerical experiments for the Rosenau equation we investigate how the finite number of grid points affects the  error. As for the BBM equation, this is done by fixing the mesh size, $h=0.05$, and increasing $N$ until the error does not decrease anymore, i.e., until the error is nearly identical to the discretization error.  In Figure \ref{fig:Fig6}, we plot, on a semi-logarithmic scale, the numerical results for the $l^{\infty}$-error $E_{h}^{N}$ at time $t=10$  with $N$. The figure shows that the error decreases exponentially for  the relatively small values of $N$ while the relatively large values of $N$  have no  major influence on the error. Recalling that the solitary wave solution in (\ref{eq:solitaryR}) decays exponentially to zero for $|x| \rightarrow \infty$ and  the expectation  $\epsilon={\cal O}\left(e^{-CNh}\right)$ obtained from Example \ref{ex:ex5.7}, one can conclude that these numerical experiments validate the theoretical results. Also, it is worth mentioning that   the above results obtained for the Rosenau equation are qualitatively similar to those obtained for the BBM equation.
\begin{figure}[h!]
    \centering
    \includegraphics[width=0.70\linewidth,scale=1.50,keepaspectratio]{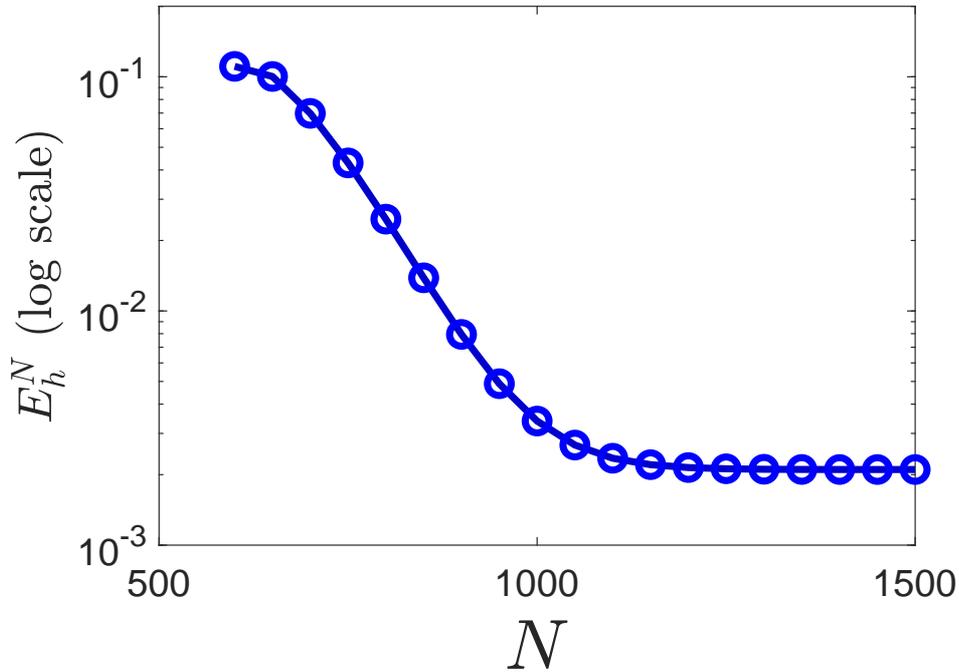}
    \caption{Variation of  the $l^{\infty}$-error ($E_{h}^{N}$) at   $t=10$ with $N$ for the  solitary wave problem of the nonlocal nonlinear wave equation (\ref{eq:cont}) with the wave speed  $c=0.5$, the kernel $\beta(x)= {1\over {2\sqrt 2}}e^{-{\vert x\vert\over \sqrt 2}}\Big( \cos\big({{\vert x\vert}\over {\sqrt 2}}\big) + \sin \big({{\vert x\vert}\over {\sqrt 2}}\big) \Big )$  and the  nonlinear term  $f(u)=u-10u^{3}+12 u^{5}$. The  mesh size  is fixed at $h=0.05$. The computational domain is $[-Nh,  Nh]$. The circle markers indicate the data points of the numerical experiments. }
    \label{fig:Fig6}
\end{figure}

\bibliographystyle{unsrt}
\bibliography{uniref}

\begin{thebibliography}{1}

\bibitem{Benjamin1972}
T.~B. Benjamin, J.~L. Bona, and J.~J. Mahony.
\newblock Model equations for long waves in nonlinear dispersive systems.
\newblock {\em Philos. Trans. R. Soc. Lond. Ser. A: Math. Phys. Sci.},
  272:47--78, 1972.

\bibitem{Rosenau1988}
P.~Rosenau.
\newblock {D}ynamics of dense discrete systems: {H}igh order effects.
\newblock {\em Prog. Theor. Phys.}, 79:1028--1042, 1988.

\bibitem{Borluk2017}
H.~Borluk and G.~M. Muslu.
\newblock Numerical solution for a general class of nonlocal nonlinear wave
  equations arising in elasticity.
\newblock {\em ZAMM-Z. Angew. Math. Mech.}, 97:1600--1610, 2017.

\bibitem{Erbay2018}
H.~A. Erbay, S.~Erbay, and A.~Erkip.
\newblock Convergence of a semi-discrete numerical method for a class of
  nonlocal nonlinear wave equations.
\newblock {\em ESAIM: Math. Model. Numer. Anal.}, 52:803--826, 2018.

\bibitem{Bona1981}
J.~L. Bona, W.~G. Pritchard, and L.~R. Scott.
\newblock {A}n evaluation of a model equation water waves.
\newblock {\em Philos. Trans. R. Soc. Lond. Ser. A: Math. Phys. Sci.},
  302:457--510, 1981.

\bibitem{Duruk2010}
N.~Duruk, H.~A. Erbay, and A.~Erkip.
\newblock Global existence and blow-up for a class of nonlocal nonlinear
  {C}auchy problems arising in elasticity.
\newblock {\em Nonlinearity}, 23:107--118, 2010.

\bibitem{Constantin2002}
A.~Constantin and L.~Molinet.
\newblock The initial value problem for a generalized {B}oussinesq equation.
\newblock {\em Differential and Integral Equations}, 15:1061--1072, 2002.

\bibitem{Bona2000}
J.~L. Bona, W.~R. McKinney, and J.~M. Restrepo.
\newblock {S}table and unstable solitary-wave solutions of the generalized
  regularized long-wave equation.
\newblock {\em J. Nonlinear Sci.}, 10:603--638, 2000.

\bibitem{Park1992}
M.~A. Park.
\newblock {P}ointwise decay estimates of solutions of the generalized {R}osenau
  equation.
\newblock {\em J. Korean Math.}, 29:261--280, 1992.

\end{thebibliography}

\end{document}